\documentclass{article}
\usepackage{authblk}
\title{Quantization Dimension of $1$-variable Random Self-Similar Measures}
\author[1]{Akash Banerjee\thanks{Corresponding author email: akash.mapping@gmail.com}}
\author[2]{Alamgir Hossain}
\author[3]{Md. Nasim Akhtar}

\affil[1,2,3]{Department of Mathematics, Presidency
University, Kolkata, 700073, West Bengal, India}

\usepackage{graphicx} % Required for inserting images
\usepackage{multirow}%
\usepackage{amsmath,amssymb,amsfonts}%
\usepackage{amsthm}%
\usepackage{mathrsfs}%
\usepackage{xcolor}%
\usepackage{textcomp}%
\usepackage{manyfoot}%
\usepackage{booktabs}%
\usepackage{algorithm}%
\usepackage{algorithmicx}%
\usepackage{algpseudocode}%
\usepackage{listings}%
\usepackage{dutchcal,physics}
\newtheorem{theorem}{Theorem}%  meant for continuous numbers
\newtheorem{proposition}[theorem]{Proposition}% 
\newtheorem{lemma}[theorem]{Lemma}% 
\newtheorem{corollary}[theorem]{Corollary}% 
\newtheorem{example}{Example}%
\newtheorem{remark}{Remark}%
\newtheorem{definition}{Definition}%

\newcommand{\pcr}[2]{\underset{\sigma \in #1}{\sum} (p_{\sigma} c_{\sigma}^r)^{\frac{#2}{r+#2}}}

\newcommand{\gmomg}{\Gamma_{\omega}}

\newcommand{\lmomgp}[1]{\Lambda_{\omega}^{(#1)}}
\newcommand{\lmbd}[2]{\Lambda_{#1}^{#2}}
\newcommand{\qncf}[3]{#1^{\frac{1}{#2}}V_{#1,r}^{\frac{1}{r}}(#3)}
\newcommand{\kr}{\kappa_r}
\newcommand{\krr}{\frac{\kappa_r}{r+\kappa_r}}

\newcommand{\Om}[1]{\Omega_{#1}}
\newcommand{\om}[1]{\omega_{#1}}

\newcommand{\N}{\mathbb{N}}
\newcommand{\Rl}{\mathbb{R}}

\newcommand{\lr}[3]{\left#1 #2 \right#3}
\newcommand{\gmomgn}[1]{\Gamma_{\om{},#1}}

\date{}

\begin{document}

\maketitle

%% Abstract
\begin{abstract}
%% Text of abstract
The quantization problem for random fractals presents unique challenges due to the lack of uniform geometric scaling inherent in deterministic systems. In this article, we establish the almost sure quantization dimension for a class of $1$-variable (homogeneously) random self-similar measures. Unlike the deterministic setting, where the dimension is derived from a fixed pressure function, we prove that in the random case, the quantization dimension $\kappa_{r}$ is the unique zero of the expectation of the topological pressure. We rigorously justify this by exploiting the ergodicity of the shift map on the symbolic space to control distortion errors across non-uniform scales. Our results highlight the thermodynamic formalism underlying the quantization of random dynamical systems. 
% Furthermore, we resolve the stability of the upper quantization coefficient, proving that it is finite almost surely if the system satisfies a \underline{`balanced weight'} condition.
% In this article, we study a class of invariant measures generated by a random homogeneous self-similar iterated function system. Unlike the deterministic setting, the random quantization problem requires controlling distortion errors across non-uniform scales. For $r>0$, under a suitable separation condition, we precisely determine the almost sure quantization dimension $\kr$ of this class by utilizing the ergodic theory of the shift map on the symbolic space. By imposing an additional separation condition, we establish almost sure positivity of the $\kr$-dimensional lower quantization coefficient.  Furthermore, without assuming any separation condition, we provide a sufficient condition that guarantees almost sure finiteness of the  $\kr$-dimensional upper quantization coefficient. We also include some illustrative examples.
\end{abstract}

\section{Introduction}\label{sec1}
 The quantization problem originated in the field of digital signal processing and information theory \cite{4,1,2,3,zador_ieee}. An example of quantization is the conversion of a continuous analog signal, like a sound wave, into a discrete digital signal. Subsequently, mathematicians explored the problem extensively over the last few decades, starting with Graf and Luschgy \cite{graf1997cantorquantization,Graf2000AsymptoticsOT,Graf2000FoundationsOQ,graf_luschgy_2002_quant_self-similar}. Mathematically, this problem deals with approximations with respect to the Wasserstein $L_r$ metric, of a given probability measure by discrete probability measures with finite support. The error arising from this approximation process is called the quantization error, and the asymptotic behaviour of this error is captured by the quantization dimension. In recent years, the theory of quantization has been undergoing rapid development, spurred by fundamental research, for instance \cite{priyadarshi2025,Graf_Luschgy_normbased_2006,marc_zhu_2015_recent_developments,
zhu_marc_trans_compactlysupported,   marc_zhu_2016_carpet, zhu_Kessebhmer_2017_UpperLowerMarkov,Lindsay_Mauldin_quantization_conformal_2002, 
Amit_etal_bi_recurrent,ZHU_2008_FConformal,ZHU2008_canor_like,
zhu_2012_note,
ZHU2021_optimalVoronoi,
ZHU2017_exact_conv_order,
Zhu_2020_Asymp_carpet}. Moreover, it is also being applied to vast areas of mathematics and other fields, such as \cite{Apl_Quant_7,Apl_Quant_3,Apl_Quant_6,Apl_Quant_5,Apl_Quant_1}.
 
     Let $r>0$ and $\mu$ be a Borel probability measure on the Euclidean space $\Rl^d$. For $n \in \N$, the $n$-th quantization error (for $\mu$) of order $r$ is defined as the following:
    \begin{equation}\label{eqn:quantization error}
        V_{n,r}(\mu):=\inf \lr{\{}{\int d(x,\alpha)^r d\mu(x): \alpha \subset \Rl^d,~ \operatorname{card}(\alpha)\leq n}{\}},
    \end{equation}
    where $d(x,\alpha):= \inf \{d(x,a): a \in \alpha\}$ and $d(x,a):=\norm{x-a}$ (where $\norm{\cdot}$ denotes the usual norm on $\Rl^d$), also $\operatorname{card}(\cdot)$ denotes the cardinality of a set. The set $\alpha \subset \Rl^d$ is called an $n$-optimal set (for $V_{n,r}(\mu)$) if the infimum in $(\ref{eqn:quantization error})$ is attained at $\alpha.$ Graf and Luschgy \cite[Theorem 4.12]{ Graf2000FoundationsOQ} gave a sufficient condition for the existence of such optimal sets, namely
    \begin{equation}\label{inq:suff optimal}
        \int \norm{x}^r d\mu(x)<\infty.
    \end{equation}
    As in this paper, we will only consider measures with compact support, (\ref{inq:suff optimal}) always holds, ensuring the existence of optimal sets for $V_{n,r}(\mu)$. The lower and upper quantization dimension of order $r$ for $\mu$ are defined respectively as:
    \begin{equation*}
        \underline{D}_r(\mu):=\liminf_{n \to \infty} \frac{r \log n}{- \log V_{n,r}(\mu)}~,~~~\overline{D}_r(\mu):=\limsup_{n \to \infty} \frac{r \log n}{- \log V_{n,r}(\mu)}.
    \end{equation*}
    In case of $\underline{D}_r(\mu)=\overline{D}_r(\mu)$, the common value is called the quantization dimension of order $r$ for $\mu$ and it is denoted by $D_r(\mu)$. 
    
    The $s$-dimensional lower and upper quantization coefficients of order $r$ are respectively denoted by $\underline{Q}_r^{s}(\mu)$ and $\overline{Q}_r^{s}(\mu)$ and they are defined by
    \begin{equation*}
        \underline{Q}_r^{s}(\mu):=\liminf_{n \to \infty}\qncf{n}{s}{\mu},~~\overline{Q}_r^{s}(\mu):=\limsup_{n \to \infty}\qncf{n}{s}{\mu}, 
    \end{equation*}
    where $s>0$. As in the case of the Hausdorff dimension, the $s$-dimensional Hausdorff measure goes from zero to infinity when $s$ crosses the dimension \cite[Section 3.2]{falconer2013fractalBook}. Similarly, the lower and upper quantization dimensions are the critical points, where the lower and upper quantization coefficients (respectively) go from zero to infinity. So if for some $s>0$, both  $\underline{Q}_r^{s}(\mu)$ and $\overline{Q}_r^{s}(\mu)$ are positive and finite then $D_r(\mu)=s$ \cite[Proposition 11.3]{Graf2000FoundationsOQ}.
    
Let $N \in \N$ and $\{S_i: i=1,\dots,~N \}$ be a set of contracting similarities on $\Rl^d$ with contraction ratios $0<c_i<1$, $i=1,\dots,~N$. According to \cite{Hutschinson_1981_Frac_Self}, there exists a unique non-empty compact subset $F$ of $\Rl^d$ such that 
    \begin{equation*}
        F=\bigcup_{i=1}^N S_i(F).
    \end{equation*}
    F is known as the attractor or in this case, the self-similar set associated with the iterated function system (IFS) $\{\Rl^d, S_i: i=1,~\dots,~N\}$. Also, if probability $0\leq p_i \leq 
    1$ is assigned to the map $S_i$ ($i=1,\dots,~N$) with $\sum_{i=1}^N p_i=1$, by \cite{Hutschinson_1981_Frac_Self} there exists a unique Borel probability measure $\mu$ on $\Rl^d$ such that $\mu$ is supported on $F$ and
    \begin{equation*}
        \mu=\sum_{i=1}^N p_i (\mu \circ S_i^{-1}).
    \end{equation*}
    This measure $\mu$ is known as the invariant measure or in this case, the self-similar measure associated with the IFS and the probability vector $(p_1,\dots,~p_N).$ This IFS satisfies the open set condition (OSC) if there exists a non-empty open set $U \subset \Rl^d$ such that 
    \begin{equation*}
        U \supset \bigcup_{i=1}^N S_i(U)
    \end{equation*}
    with the union is disjoint. For $r>0$, let $\mathcal{d}_r >0 $ be given by 
    \begin{equation*}
        \sum_{i=1}^N (p_ic_i^r)^{\frac{\mathcal{d}_r}{r+\mathcal{d}_r}}=1.
    \end{equation*}
    Graf and Luschgy \cite{Graf2000AsymptoticsOT} showed that if the above IFS satisfies the OSC, then $D_r(\mu)=\mathcal{d}_r$.
    
    A random iterated function system (RIFS) is a collection of a finite number of deterministic iterated function systems (IFSs). Various combinations of these IFSs can generate a continuum of attractors and invariant measures. This allows us to study the properties (usually dimensional) of a typical attractor or measure within the continuum. There are various ways to interpret the term `typical' here. In this paper, we use a measure-theoretic framework to achieve that. 
    
    Let us denote a RIFS by $\mathcal{I}:=\{\mathcal{I}_i: i=1,\dots,~N\},$ where each $\mathcal{I}_i$ represents a deterministic IFS, namely $\mathcal{I}_i:=\{S_{i,j}:\Rl^d \to \Rl^d \mid j \in \mathbf{I}_i \},$ where $\mathbf{I}_i$ is a finite index set with $\operatorname{card}(\mathbf{I}_i)>1$. This paper focuses on $S_{i,j}$'s as contracting similarities with $0<c_{i,j}<1$ as similarity ratios and they are considered self-maps on a non-empty compact subset of $\Rl^d$, say $X$. That is for $i\in \lmbd{}{}$ and $j \in \mathbf{I}_i$, $S_{i,j}:X \to X$ satisfies 
    \begin{equation}\label{eq: self-similar RIFS}
        \norm{S_{i,j}(x)-S_{i,j}(y)}=c_{i,j}\norm{x-y}
    \end{equation}
    for all $x,y \in X.$

    Before continuing, it is imperative to introduce certain notations. 
    
    Set $\Lambda:=\{1,\dots,~N\}$ and $\Omega:=\Lambda^\N=\{\omega=(\omega_1,~\omega_2,\dots): \omega_i \in \Lambda
    \}$. Let $ \om{}=(\omega_1,~\omega_2,\dots) \in \Om{}.$ The random self-similar set or the attractor associated with $\om{}$ is defined by
    \begin{equation}\label{def: Random Attractor}
    F_\omega := \bigcap_{n \in \mathbb{N}}~~ \bigcup_{i_1 \in \mathbf{I}_{\omega_1},
    \dots, i_n \in \mathbf{I}_{\omega_n}} S_{\omega_1,i_1} \circ \dots \circ S_{\omega_n,i_n} (X).
\end{equation}
Now, assign probability $p_{i,j}>0$ to the map $S_{i,j}$ in a way that for each $i \in \lmbd{}{}$, we have $\sum_{j \in \mathbf{I}_i} p_{i,j}=1$. Then corresponding to each $\om{} \in \Om{}$ there exists a unique Borel probability measure $\mu_{\om{}}$ on $X$, with support $F_{\om{}}$, such that 
    \begin{equation}\label{def: Random Hom. self-similar measure}
        \mu_\omega:= \lim_{n \to \infty} \sum_{i_1 \in \mathbf{I}_{\omega_1},
    \dots, i_n \in \mathbf{I}_{\omega_n}} (p_{\omega_1,i_1}\dots p_{\omega_n,i_n}) \nu \circ (S_{\omega_1,i_1} \circ \dots \circ S_{\omega_n,i_n})^{-1},
    \end{equation}
    where $\nu$ is any arbitrary Borel probability measure on $X$. Here, the limit is taken in the sense of the Monge-Kantorovich metric (for details on this metric, refer to \cite{Hutschinson_1981_Frac_Self}). Proofs of existence and uniqueness of $\mu_{\om{}}$ are given in Section \ref{sec:Invariant Measure}.  
    
    The main goal of this paper is to estimate the quantization dimension of $\mu_{\om{}}$ for a typical $\om{} \in \Om{}$. Here, the term `typical' refers to `almost surely with respect to a naturally defined probability measure' on the space $\Om{}$. To construct such a natural probability measure first we assume that $\zeta:=(\zeta_1,\dots,~\zeta_N)$ be a probability vector with $\zeta_i>0$, which assigns probability $\zeta_i$ to the IFS $\mathcal{I}_i$. Set $\lmbd{}{*}:=\{\mathcal{v}=(\mathcal{v}_1,\dots,~\mathcal{v}_n): \mathcal{v}_i \in \lmbd{}{},~ n\in \N\}$ and for $\om{} \in \Om{}$ and $n \in \N$, set $\om{}\mid_n:=(\om{1},\dots,~\om{n})$. For $\mathcal{v} \in \lmbd{}{*}$, we define
    \begin{equation*}
        \mathcal{C}_n(\mathcal{v}):=\{\om{} \in \Om{}: \om{}\mid_n=\mathcal{v}\},
    \end{equation*}
    as the cylinder sets in $\Om{}$. Then
    there exists a Borel probability measure $\mathbf{P}$ on $\Om{}$ which assigns probability $\prod_{i=1}^n \zeta_{\mathcal{v}_{i}}$ to the cylinder set $\mathcal{C}_n(\mathcal{v})$. This probability measure is known as Bernoulli measure on the symbolic space $\Om{}$ and it is invariant and ergodic corresponding to the left shift map $\mathcal{L}(\om{1},~\om{2},~\om{3},\dots):=(\om{2},~\om{3},\dots)$ on $\Om{}$. Barnsley et al. \cite{BarnsleyHutchinson_v_dim_2012} and Fraser et al. \cite{FRASER_MIAO_TROSCHEIT_2018_Assouad} used this measure in the context of almost sure Hausdorff dimension and almost sure Assouad dimension (respectively) in the case of self-similar RIFSs. Hare et al. \cite{Sascha_Local}  used this concept of randomness to determine almost sure local dimensions in a self-similar setting. Also, Fraser et al. \cite{Fraser_Olsen_Multifractal} and Gui et al. \cite{Gui_Li_RandomCarpet_2008} used it in the self-affine setup.  For more on this, see \cite{Hambly1992,Olsen2011_randomSponge,olsen2017random_book,troscheit_2017_RandomGraph}. %\textcolor{red}{On the other hand, Dai et al. \cite{DAI20104_Random_Quantization} used a different concept of randomness to estimate the quantization dimension of a random self-similar measure.}
    
The extension from deterministic to $1$-variable random self-similar measures introduces a fundamental obstruction not present in the classic works of Graf and Luschgy (\cite{Graf2000FoundationsOQ}). In the deterministic setting, the measure $\mu$ is strictly self-similar, satisfying the identity $\mu(S_i(A)) = p_i \mu(A)$ for any Borel set $A$. In contrast, for the $1$-variable random measures $\mu_\omega$ considered here, the strict self-similarity is broken (see Proposition \ref{prop: mu om= sum over antichain}). Instead, the measure satisfies the relation $\mu_\omega(S_{\omega_1}(A)) = p_{\omega_1} \mu_{\mathcal{L}(\omega)}(A)$, where $\mathcal{L}(\omega)$ is the shift of the random sequence. Since $\mu_{\mathcal{L}(\omega)} \neq \mu_\omega$ in general, the standard inductive techniques for estimating the quantization dimension fail. To overcome this, we must lift the analysis to the symbolic space $\Omega$, utilizing the ergodicity of the shift map to control the asymptotic behavior of the error along typical realisations.

        Some additional definitions and notations are required to proceed further into the discussion. For $\om{} \in \Om{}$ and $n \in \N$, set $\lmbd{\om{}}{(n)}:=\{\sigma=(\sigma_1,\dots,~\sigma_n): \sigma_j \in \mathbf{I}_{\om{j}} ~\text{for}~j=1,\dots,~n\}$ and $\lmbd{\om{}}{*}:=\bigcup_{n \in \N} \lmomgp{n}$. Now for $\sigma \in \lmomgp{n}$, set $\abs{\sigma}:=n$ and for $j \leq n $, set $\sigma \mid_j:=(\sigma_1,\dots,~\sigma_j)$. For $n>1$, set $\sigma^-:=\sigma \mid_{(n-1)}$. Also, define 
        \begin{align*}
            S_{\sigma} &:= S_{\om{1},\sigma_1} \circ \dots \circ S_{\om{n},\sigma_n},~~c_{\sigma} := c_{\om{1},\sigma_1} \dots c_{\om{n},\sigma_n},\\
            p_{\sigma} &:= p_{\om{1},\sigma_1} \dots p_{\om{n},\sigma_n},~~
            E_{\sigma} := S_{\sigma}(X).
        \end{align*}
For $\sigma,~ \tau \in \lmbd{\om{}}{*}$, $\sigma$ is called a predecessor of $\tau$ if $\abs{\sigma}\leq \abs{\tau}$ and $\tau \mid_{\abs{\sigma}}=\sigma$ and it is denoted by $\sigma \preccurlyeq \tau$. On the other hand $\sigma$ is called a successor of $\tau$ if $\abs{\sigma}\geq \abs{\tau}$ and $\sigma \mid_{\abs{\tau}}=\tau$ and it is denoted by $\sigma \succcurlyeq \tau$. A strict symbol ($\sigma \prec \tau$ or $\sigma \succ \tau$) is used in the case of $\abs{\sigma} \neq \abs{\tau}$. Also for $\sigma\neq \tau$, they are called incomparable if neither $\sigma \prec \tau$ nor $\sigma \succ \tau$ holds. For $\sigma \in \lmomgp{n}$ and $j \in \N$, we
define 
\begin{equation*}
    \lmbd{j}{}(\sigma) :=\{\tau \in \lmomgp{n+j}: \sigma \prec \tau\}.
\end{equation*}

Here, we introduce a separation condition, which is a random analogue of the ESSC given in \cite{ZHU2008_canor_like}. This condition determines the extent to which various parts of the attractor intersect or overlap with each other. It is defined as follows.

\begin{definition}[UESSC] We say that the RIFS $\mathcal{I}$ satisfies the uniform extra strong separation condition (UESSC) if there exists $\beta >0$ such that for each $i \in \lmbd{}{}$, we have
    \begin{equation}\label{inq:original UESSC}
        \min\left\{\operatorname{dist}(E_{i,j},E_{i,j'}): j\neq j' \in \mathbf{I}_i\right\} \geq \beta \cdot \max \left\{\abs{E_{i,j}}: j \in \mathbf{I}_i\right\},
    \end{equation} where $E_{i,j}:=S_{i,j}(X)$, $\abs{E_{i,j}}$ denotes the diameter of the set $E_{i,j}$ and $\operatorname{dist}(A,B)$ denotes the usual distance between two sets $A,B \subset \Rl^d$ with respect to the usual metric on $\Rl^d$.
\end{definition}

Note that, if the RIFS $\mathcal{I}$ satisfies the UESSC then it follows that for any $\sigma \in \lmbd{\om{}}{*},$ 
\begin{equation}\label{inq:applicable UESSC}
    \min\left\{\operatorname{dist}(E_{\rho},E_{\tau}): \rho \neq \tau \in \Lambda_1(\sigma) \right\} \geq \beta \cdot \max \left\{\abs{E_{\tau}}: \tau \in \Lambda_1(\sigma)\right\}.
\end{equation}
This implication of UESSC will be employed more frequently in our proofs than the original formulation (\ref{inq:original UESSC}).

A central role in our analysis is played by the `expected pressure function' associated with the random system. The following proposition establishes the existence and uniqueness of the critical exponent $\kappa_r$, which will later be identified as the almost-sure quantization dimension.

\begin{proposition}\label{prop: existence of k_r}
    For $r>0$ there exists a unique $\kr >0 $ such that
    \begin{equation*}
        \sum_{j=1}^N \zeta_j \cdot \log \sum_{k \in \mathbf{I}_j}(p_{j,k}c_{j,k}^r)^\krr = 0,
    \end{equation*}
    equivalently
    \begin{equation}\label{eqn:def of kr prod}
        \prod_{j=1}^N \lr{[}{\sum_{k \in \mathbf{I}_j}(p_{j,k}c_{j,k}^r)^\krr }{]}^{~\zeta_j}=1.
    \end{equation}
\end{proposition}
\begin{proof}
    For $z\geq 0$, define \[T(z)=\sum_{j=1}^N \zeta_j \cdot \log \sum_{k \in \mathbf{I}_j}(p_{j,k}c_{j,k}^r)^z.\]
Clearly, $T$ is differentiable for $z\geq0$ and 
\[T'(z)=\sum_{j=1}^N \zeta_j \cdot \frac{\sum_{k \in \mathbf{I}_j}(p_{j,k}c_{j,k}^r)^z \cdot \log (p_{j,k}c_{j,k}^r)}{\sum_{k \in \mathbf{I}_j}(p_{j,k}c_{j,k}^r)^z} <0, \]
since $0<p_{j,k}c_{j,k}^r < 1$ for $j=1,\dots,N.$\\

Note that $T(0)=\sum_{j=1}^N \zeta_j \cdot \log(\operatorname{card}(\mathbf{I}_j))>0$ as $\operatorname{card}(\mathbf{I}_j)>1$ for $j=1,\dots,N.$ Also, $T(1)=\sum_{j=1}^N \zeta_j \cdot \log \sum_{k \in \mathbf{I}_j}(p_{j,k}c_{j,k}^r)<0$ as $\sum_{k \in \mathbf{I}_j}(p_{j,k}c_{j,k}^r)<\sum_{k \in \mathbf{I}_j}p_{j,k}=1$ for $j=1\dots,N.$\\

So there exists unique $0<z_0<1$ such that $T(z_0)=0$. Setting $\kr:=\frac{rz_0}{1-z_0}$ the result follows.
\end{proof}

With the necessary groundwork in place, we can now formally state the main theorem of this paper. See Section \ref{sec:proofs} for proofs.
 
\begin{theorem}\label{thm: main}
    Let $\mathcal{I}$ be the RIFS as defined in (\ref{eq: self-similar RIFS}), consisting of similarity maps and $\mu_{\om{}}$ be the $1$-variable random self-similar measure given in (\ref{def: Random Hom. self-similar measure}). Let $r>0$ and $\kr$ be the unique positive real number given in Proposition \ref{prop: existence of k_r}. If $\mathcal{I}$ satisfies the UESSC then for $\mathbf{P}$-almost all $\om{} \in \Om{}$, we have
        \begin{equation*}
            D_r(\mu_{\om{}})=\kr.
        \end{equation*}
        
        % \item Along with the UESSC if $\mathcal{I}$ satisfies the SUOSC then for $\mathbf{P}$-almost all $\om{} \in \Om{}$, we have
        % \begin{equation*}
        %      \underline{Q}_{r}^{\kr}(\mu_{\om{}})>0.
        % \end{equation*}
\end{theorem}

\begin{remark}
    It is worth highlighting here the thermodynamic interpretation of our main result. In the deterministic setting, the quantization dimension is the unique zero of the topological pressure function $P(t) = \log \sum_{i} (p_i c_i^r)^t$. Our result establishes that for $1$-variable random self-similar measures, the quantization dimension $\kappa_r$ is determined by the zero of the expectation of the topological pressure:
    $$\mathbb{E}_{\mathbf{P}}\left[ \log \sum_{j \in I_{\omega_1}} (p_{\omega_1, j} c_{\omega_1, j}^r)^{\frac{s}{r+s}} \right] = 0,~\text{where}~s>0.$$
    Here, the ergodicity of the Bernoulli measure on the symbolic space ensures that the geometric fluctuations average out almost surely, allowing for a precise dimensional formula that generalises the deterministic case in a natural yet non-trivial thermodynamic framework.
\end{remark}

\begin{remark}
    For $n \in \N$, we write $\mathcal{L}^n=\mathcal{L} \circ \mathcal{L} \circ \dots \circ \mathcal{L}$, where the composition is taken $n$-times and $\mathcal{L}^0=\mathcal{L}.$
\end{remark}

\begin{remark}
For the remainder of this paper, unless explicitly stated otherwise, $\om{}=(\omega_1,\omega_2,\dots)$ is an arbitrary element in $\Om{}$ and the phrase  `almost all $\om{}$' will refer to `$\mathbf{P}$-almost all $\om{} \in \Om{}$'.
\end{remark}
We now present an example of a RIFS that satisfies the separation condition UESSC.

\begin{example}\label{exmpl: 1st}
Let $X:=[0,1]$ and for $i,j \in \{1,2\}$, define $S_{i,j}:X \rightarrow \Rl$ as 
\begin{align*}
    S_{1,1}(x):=x/5+1/5,~ S_{1,2}(x):=x/5+3/5;\\ S_{2,1}(x):=x/5+1/6,~S_{2,2}(x):=x/5+3/6.
\end{align*}
Now, define IFSs $\mathcal{I}_1:=\{S_{1,1},S_{1,2}\}$, $\mathcal{I}_2:=\{S_{2,1},S_{2,2}\}$ and RIFS $\mathcal{I}:=\{\mathcal{I}_1,\mathcal{I}_2\}$. 
% For $i=1,2$, we have $ {\bigcup_{j\in\mathbf{I}_i}S_{i,j}}(X)\subset(0,1),$ where the unions are disjoint. Also by (\ref{def: Random Attractor}), $F_{\om{}}\subset\bigcup_{j\in\mathbf{I}_{\om{1}}}S_{\om{1},j}(X)$. So by taking $U=(0,1)$, we have $F_{\om{}}\cap U\neq \emptyset$ for all $\om{} \in \Om{}$.  Hence $\mathcal{I}$ satisfies the SUOSC.
We have $\operatorname{dist}(E_{1,1},E_{1,2})=1/5,~\operatorname{dist}(E_{2,1},E_{2,2})=2/15$ and $\abs{E_{i,j}}=1/5$ for all $i,j$. Hence $\mathcal{I}$ also satisfies the UESSC if we take $\beta=1/3$.
\end{example}

    \section{Existence and Uniqueness of $1$-variable Random Self-similar 
measures}\label{sec:Invariant Measure}
Let $\mathcal{M}(X)$ be the set of all Borel probability measures on $X$. Fix $\nu \in \mathcal{M}(X)$. For $n \in \mathbb{N}$, define
\begin{align*}
    \mu_{\omega, \nu}^{(n)}&:=\sum_{i_1 \in \mathbf{I}_{\omega_1},
    \dots, i_n \in \mathbf{I}_{\omega_n}} (p_{\omega_1,i_1}\dots p_{\omega_n,i_n}) \nu \circ (S_{\omega_1,i_1} \circ  \dots \circ S_{\omega_n,i_n})^{-1}\\
    &= \sum_{\sigma \in \lmomgp{n}} p_{\sigma}(\nu \circ S_{\sigma}^{-1}).
\end{align*}
Clearly, $\mu_{\omega, \nu}^{(n)} \in \mathcal{M}(X)$. The measure $\mu_{\om{}}$ is defined as the unique limit of the measures $\mu_{\omega, \nu}^{(n)}$ with respect to the Monge-Kantorovich metric $L$ on $\mathcal{M}(X)$, which is defined by
\begin{align*}
    L(\mu,\nu):=\sup\bigg\{ &\int \phi ~d\mu - \int \phi~ d\nu \mid \phi: X \to \mathbb{R},~ \\
    &\operatorname{Lip}(\phi)=\sup_{x \neq y}\frac{d(\phi(x),\phi(y))}{d(x,y)}~\leq 1 \bigg\}.
\end{align*}
Now, we show that for $\nu \in \mathcal{M}(X)$, $\{\mu_{\omega, \nu}^{(n)}\}_n$ is a Cauchy sequence in $(\mathcal{M}(X),L)$. 

Before continuing, we define a necessary notation here: for $n,~n'\in \N$ and $\sigma \in \lmbd{\om{}}{(n)},~\sigma' \in \lmbd{\mathcal{L}^n(\om{})}{(n')}$, define $\sigma \sigma':= (\sigma_1,\dots,~\sigma_n,~\sigma'_1,\dots,~\sigma'_{n'}) \in \lmomgp{n+n'}$.

Let $\phi: X \to \mathbb{R}$ be such that $\operatorname{Lip}(\phi)\leq 1$ and $j \in \mathbb{N}$. Then we have  
\begin{align*}
     &\int \phi~d\mu_{\omega, \nu}^{(n+j)} - \int \phi~d\mu_{\omega, \nu}^{(n)} \\
     &= \sum_{\tau \in \lmomgp{n+j}} p_{\tau} \int \phi~ d(\nu \circ S_{\tau}^{-1}) - \sum_{\sigma \in \lmomgp{n}} p_{\sigma} \int \phi~ d(\nu \circ S_{\sigma}^{-1})\\
     &= \sum_{\sigma \in \lmomgp{n}} p_{\sigma}\left\{\sum_{\sigma' \in \lmbd{\mathcal{L}^{n}(\om{})}{(j)}} p_{\sigma'}  \left( \int \phi~ d \left(\nu \circ S_{\sigma \sigma'}^{-1} \right)  - \int \phi~ d \left(  \nu \circ S_{\sigma}^{-1} \right)\right)\right\}\\
     &= \sum_{\sigma \in \lmomgp{n},~\sigma' \in \lmbd{\mathcal{L}^{n}(\om{})}{(j)}} p_{\sigma \sigma'}\left\{ \int \phi~ d \left(\nu' \circ S_{\sigma}^{-1} \right)  - \int \phi~ d \left(  \nu \circ S_{\sigma}^{-1} \right)\right\},
\end{align*}
where $\nu'=\nu \circ S_{\sigma'}^{-1}.$

Now,
\begin{align*}
    &\sum_{\sigma \in \lmomgp{n},~\sigma' \in \lmbd{\mathcal{L}^{n}(\om{})}{(j)}} p_{\sigma \sigma'}\left\{ \int \phi~ d \left(\nu' \circ S_{\sigma}^{-1} \right)  - \int \phi~ d \left(  \nu \circ S_{\sigma}^{-1} \right)\right\}\\
    &= \sum_{\sigma \in \lmomgp{n},~\sigma' \in \lmbd{\mathcal{L}^{n}(\om{})}{(j)}} p_{\sigma \sigma'}c_{\sigma} \lr{(}{\int (c_{\sigma}^{-1} \cdot \phi \circ S_{\sigma}^{})d(\nu')-\int (c_{\sigma}^{-1}\cdot \phi \circ S_{\sigma}^{})d(\nu)}{)} \\
    &\leq (c_{\max})^n \sum_{\sigma \in \lmomgp{n},~\sigma' \in \lmbd{\mathcal{L}^{n}(\om{})}{(j)}} p_{\sigma\sigma'
    } L(\nu',\nu),
\end{align*}
since $\operatorname{Lip}(c_{\sigma}^{-1}\cdot \phi \circ S_{\sigma}^{})\leq 1$ and we write $c_{\max}:=\max\{c_{i,j}:j \in \mathbf{I}_{i},\\
~i\in \lmbd{}{}\}$~ $(\Rightarrow0<c_{\max}<1)$. Hence
\begin{equation}\label{inq: inv. measure cauchy 1}
    \int \phi~d\mu_{\omega, \nu}^{(n+j)} - \int \phi~d\mu_{\omega, \nu}^{(n)}\leq (c_{\max})^n \sum_{\sigma \in \lmomgp{n},~\sigma' \in \lmbd{\mathcal{L}^{n}(\om{})}{(j)}} p_{\sigma\sigma'} L(\nu',\nu).
\end{equation}
Now for $\sigma'=(\sigma'_{n+1},\dots, \sigma'_{n+j}) \in \lmbd{\mathcal{L}^{n}(\om{})}{(j)}$, we have
\begin{align*}
    &L(\nu,\nu')\\
    &= L(\nu, \nu \circ S_{\sigma'}^{-1})\\
    &\leq L(\nu,\nu \circ S_{\omega_{n+1},\sigma'_{n+1}}^{-1})+ L(\nu \circ S_{\omega_{n+1},\sigma'_{n+1}}^{-1}, \nu \circ (S_{\omega_{n+1},\sigma'_{n+1}}\circ S_{\omega_{n+2},\sigma'_{n+2}})^{-1})\\&+\dots + L(\nu \circ (S_{\omega_{n+1},\sigma'_{n+1}}\circ \dots \circ S_{\omega_{n+j-1},\sigma'_{n+j-1}})^{-1}, \\
    &\nu \circ (S_{\omega_{n+1},\sigma'_{n+1}}\circ \dots  \circ S_{\omega_{n+j},\sigma'_{n+j}})^{-1})\\
    & \leq L(\nu,\nu \circ S_{\omega_{n+1},\sigma'_{n+1}}^{-1})+c_{max} L(\nu,\nu \circ S_{\omega_{n+2},\sigma'_{n+2}}^{-1})+ \dots \\
    &+ c_{max}^{j-1} L(\nu,\nu \circ S_{\omega_{n+j},\sigma'_{n+j}}^{-1})\\
    & \leq (1+ c_{\max}+\dots+c_{\max}^{j-1}) A_{\nu} \leq \frac{1}{1-c_{max}}A_{\nu},
    \end{align*} 
    where $A_\nu=\max\{L(\nu,\nu \circ S_{i,j}^{-1}): j \in \mathbf{I}_i,~i \in \lmbd{}{}\}$. So, $0\leq A_{\nu}<\infty$.\\

Putting this in $(\ref{inq: inv. measure cauchy 1})$, we get
\begin{align*}
    &\int \phi~d\mu_{\omega, \nu}^{(n+j)} - \int \phi~d\mu_{\omega, \nu}^{(n)}\\
    &\leq \frac{c_{\max}^n~ A_{\nu}}{1-c_{\max}} \sum_{\sigma \in \lmomgp{n},~\sigma' \in \lmbd{\mathcal{L}^{n}(\om{})}{(j)}} p_{\sigma \sigma'} = \frac{A_\nu}{1-c_{\max}} c_{\max}^n.
\end{align*}
Since $0<c_{max}<1$, for any $\epsilon > 0$, $j \in \mathbb{N}$ and large enough $n \in \mathbb{N}$, we have 
\begin{equation*}
    \int \phi~d\mu_{\omega, \nu}^{(n+j)} - \int \phi~d\mu_{\omega, \nu}^{(j)} < \epsilon \Rightarrow L(\mu_{\omega,\nu}^{(n+j)}, \mu_{\omega,\nu}^{(n)})\leq \epsilon.
\end{equation*}
Hence $\{\mu_{\omega, \nu}^{(n)}\}_n$ is Cauchy in $(\mathcal{M}(X),L)$. As $(\mathcal{M}(X),L)$ is complete \cite{Hutschinson_1981_Frac_Self}, there exists $\mu_{\omega,\nu} \in \mathcal{M}(X)$ such that $\{\mu_{\omega, \nu}^{(n)}\}_n$ converges to $\mu_{\omega,\nu}$.

Next, we show that for any $\nu \in \mathcal{M}(X)$, the sequence $\{\mu_{\omega, \nu}^{(n)}\}_n$ always converges to a common limit $\mu_\omega$, independent of $\nu$. To show this, let us take $\nu, \nu' \in \mathcal{M}(X)$. We will show that $L(\mu_{\omega,\nu}~,~\mu_{\omega,\nu'})=0$ and hence $\mu_{\omega,\nu}=\mu_{\omega,\nu'}$.\\

Since $\mu_{\omega,\nu}^{(n)} \overset{n}{\rightarrow} \mu_{\omega,\nu}$, we have $\int \phi~ d\mu_{\omega,\nu}^{(n)} \overset{n}{\rightarrow} \int \phi~d\mu_{\omega,\nu}$ \cite{Hutschinson_1981_Frac_Self}. That is $\int \phi ~d(\lim_{n \to \infty}\mu_{\omega,\nu}^{(n)})=\lim_{n \to \infty}\int \phi ~ d\mu_{\omega,\nu}^{(n)}$. So, we have

\begin{align}\label{uniqueness 1}
  \nonumber &\int \phi~d\mu_{\omega, \nu} - \int \phi~d\mu_{\omega, \nu'}\\
  \nonumber &= \int \phi ~d(\lim_{n \to \infty}\mu_{\omega,\nu}^{(n)})- \int \phi ~d(\lim_{p \to \infty}\mu_{\omega,\nu'}^{(p)})\\ &= \lim_{n \to \infty} \left[\int \phi ~d(\mu_{\omega,\nu}^{(n)})-\int \phi ~d(\mu_{\omega,\nu'}^{(n)})\right].
    \end{align}
Now,
\begin{align}\label{uniqueness 2}
  \nonumber  &\int \phi ~d(\mu_{\omega,\nu}^{(n)})-\int \phi ~d(\mu_{\omega,\nu'}^{(n)}) 
    \\
    \nonumber &= \sum_{\sigma \in \lmbd{\om{}}{(n)}} p_{\sigma}c_{\sigma} \left[  \int (c_{\sigma}^{-1} \cdot \phi \circ S_{\sigma}^{})~ d\nu- \int(c_{\sigma}^{-1} \cdot \phi \circ S_{\sigma}^{}) ~d\nu' \right]\\
    &\leq c_{max}^n \sum_{\sigma \in \lmbd{\om{}}{(n)}} p_{\sigma} L(\nu,\nu') \leq c_{max}^n~L(\nu,\nu').
\end{align}
Combining (\ref{uniqueness 1}) and (\ref{uniqueness 2}), we get
\begin{align*}
    \int \phi~d\mu_{\omega, \nu} - \int \phi~d\mu_{\omega, \nu'}\leq  \left(\lim_{n \to \infty} c_{max}^n\right)L(\nu,\nu')=0.
\end{align*}
Since this is true for all such $\phi$, we conclude that $L(\mu_{\omega, \nu},\mu_{\omega, \nu'})=0$.
Hence we can write, for any $\nu \in \mathcal{M}(X)$
\begin{equation}\label{def:mu_omega}
    \mu_{\omega}:= \lim_{n \to \infty} \mu_{\omega,\nu}^{(n)}=\lim_{n \to \infty}\sum_{\sigma \in \lmomgp{n}} p_{\sigma}(\nu \circ S_{\sigma}^{-1}).
\end{equation}
We call this $\mu_\omega$ the \textit{$1$-variable random self-similar measure} corresponding to $\omega$. In addition, it's worthwhile to note that the $1$-variable random measures fall under the broader category of $V$-variable measures, studied extensively in \cite{BARNSLEY2008_V-variable}. 

For the convenience of proving some of the important results in Section 
\ref{sec:proofs}, we easily derive $\mu_{\om{}}$ from (\ref{def:mu_omega}), as 
\begin{equation}\label{mu def as sascha}
    \mu_{\om{}}=\sum_{i_1 \in \mathbf{I}_{\om{1}}}p_{\om{1},i_1} (\mu_{\mathcal{L}(\om{})} \circ S_{\om{1},i_1}^{-1})
\end{equation}
  and hence for any $n \in \N$
\begin{equation}\label{mu def sascha (n)}
    \mu_{\om{}}=\sum_{\sigma \in \lmbd{\om{}}{(n)}} p_{\sigma} (\mu_{\mathcal{L}^n(\om{})} \circ S_{\sigma}^{-1}).
\end{equation}
In \cite{Sascha_Local},  $\mu_{\om{}}$ has been considered of the form (\ref{mu def as sascha}) to study its multifractal analysis.
%In the sequel, we consider $\mu_{\om{}}$ as given in ($\ref{mu def sascha (n)}$).

\begin{definition}
     A finite set $\Gamma_{\om{}} \subset \lmbd{\om{}}{*}$ is called a finite maximal antichain (FMA) if any $\sigma \neq \sigma' \in \Gamma_{\om{}}$ are incomparable and any $\tau \in \lmbd{\om{}}{*}$ is comparable with some $\sigma \in \Gamma_{\om{}}.$
\end{definition}

The following proposition shows that although $\mu_{\om{}}$ may not always be strictly self-similar, some sort of partial self-similarity is still involved.
\begin{proposition}\label{prop: mu om= sum over antichain}
    For any finite maximal antichain $\gmomg \subset \lmbd{\om{}}{*}$, we have
    \begin{equation}\label{eqn: mu_omega on FMA}
        \mu_{\om{}}=\sum_{\sigma \in \gmomg} p_{\sigma} (\mu_{\mathcal{L}^{\abs{\sigma}}(\om{})} \circ S_{\sigma}^{-1}).
    \end{equation}
    \end{proposition}
    \begin{proof}
    For $n \in \N$, set $\Gamma_{n}:=\{\sigma \in \gmomg : \abs{\sigma}=n\}.$ Since $\gmomg$ is a FMA there exists $n_1<\dots < n_k \in \N$ such that \[\displaystyle{\gmomg=\bigcup_{i=1}^k\Gamma_{n_i}}\] for some $k \in \N$.
    
    For $i=1,\dots,k$ set $$\Gamma_{n_i}':=\left\{\tau \in \Lambda_{\om{}}^{(n_k)}: \sigma \preccurlyeq \tau~\text{for some}~\sigma 
    \in \Gamma_{n_i} \right\}.$$ Then using $(\ref{mu def sascha (n)})$, we deduce
    \begin{equation}\label{prop2.2 (1)}
        \sum_{\sigma \in \Gamma_{n_i}}  p_{\sigma} (\mu_{\mathcal{L}^{n_i}(\om{})} \circ S_{\sigma}^{-1})=\sum_{\sigma \in \Gamma_{n_i}'}  p_{\sigma} (\mu_{\mathcal{L}^{n_k}(\om{})} \circ S_{\sigma}^{-1}).
    \end{equation}
    % For $i=2,\dots,k$, $\Gamma_{n_i}^{(n_1)}:=\{\tau \in \lmomgp{n_1}: \tau \prec \sigma ~\text{for some}~ \sigma \in \Gamma_{n_i} \}$ and $\Gamma_{n_1}^{(n_1)}:=\Gamma_{n_1}.$
% Note that, since $\gmomg$ is a maximal antichain, \[\lmomgp{n_1}=\bigcup_{i=1}^k \Gamma_{n_i}^{(n_1)}.\]
Now
    \begin{align*}
        &\sum_{\sigma \in \gmomg} p_{\sigma} (\mu_{\mathcal{L}^{\abs{\sigma}}(\om{})} \circ S_{\sigma}^{-1})\\
        &=\sum_{i=1}^k \sum_{\sigma \in \Gamma_{n_i}}  p_{\sigma} (\mu_{\mathcal{L}^{n_i}(\om{})} \circ S_{\sigma}^{-1})\\
        &= \sum_{i=1}^k \sum_{\tau \in \Gamma_{n_i}'}  p_{\tau} (\mu_{\mathcal{L}^{n_k}(\om{})} \circ S_{\tau}^{-1})~~~~(\text{by (\ref{prop2.2 (1)})})\\
        &= \sum_{\tau \in \lmomgp{n_k}}  p_{\tau} (\mu_{\mathcal{L}^{n_k}(\om{})} \circ S_{\tau}^{-1})~~~~(\text{since}~\Gamma_{\om{}}~\text{is a FMA})\\
        &=\mu_{\om{}}~~~~(\text{by (\ref{mu def sascha (n)})}).
    \end{align*}
    Hence, the result follows.
\end{proof}

\section{Preliminary facts}\label{sec:def not}

First, we introduce `periodic words' in $\Om{}$. For $n \in \mathbb{N}$, write
    \[\displaystyle{\omega_n^p:=(\omega_1,\omega_2,\dots,\omega_n,\omega_1,\omega_2,\dots,\omega_n,\omega_1, \omega_2,\dots,\omega_n,\dots) \in \Omega}.\] Here the `$p$' in $\om{n}^p$ highlights the periodic nature of $\om{n}^p$. We will refer to $\omega_n^p$ as a periodic word with period $n$. Various results emerging from this periodicity can be established in suitable contexts. Some of them are given below.
    \begin{enumerate}
        \item If we define a metric $\tilde{d}$ on $\Om{}$ by $\tilde{d}(\om{},\om{}'):=2^{-(\om{} \wedge \om{}')}$, where $\om{} \wedge \om{}':=\max\{k \in \N: \om{j}=\om{j}',~j=1,\dots,~k\}$ and $\tilde{d}(\om{},\om{}')=0$ if $\om{}=\om{}'$, then $\om{n}^p \overset{n}{\to} \om{}$ in $(\Om{},\tilde{d})$.
        \item  With respect to the metric `$L$' defined in Section $\ref{sec:Invariant Measure}$, $\mu_{\om{n}^p} \overset{n}{\to} \mu_{\om{}}.$
        \item For $k \in \N$, $V_{k,r}(\mu_{\om{n}^p}) \overset{n}{\to} V_{k,r}(\mu_{\om{}}),$ where the convergence is uniform over $k$.
        \item For almost all $\om{}$, $D_r(\mu_{\om{n}^p}) \overset{n}{\to}D_r(\mu_{\om{}})$, provided $\mathcal{I}$ satisfies the UESSC.
        \end{enumerate}
        
        The above first three statements are easy to verify. The fourth one can be justified from Propositions \ref{lemma:Strng law of large numbers} and \ref{prop: D_r=k_r}.
        
        Here, we show that $\mu_{\om{n}^p}$~, the $1$-variable random measure corresponding to $\om{n}^p$ is a self-similar measure for the IFS $\mathbf{I}_{\om{},n}:=\{S_{\sigma}: \sigma \in \lmomgp{n} \}.$ To see this, assign probability $p_{\sigma}>0$ to $S_{\sigma}$. Define $\theta:\mathcal{M}(X) \to \mathcal{M}(X)$ by 
        \begin{equation*}
            \theta(\nu)=\sum_{\sigma \in \lmomgp{n}} p_{\sigma}(\nu \circ S_{\sigma}^{-1}),
        \end{equation*}
        where $\nu \in \mathcal{M}(X).$ Then by \cite{Hutschinson_1981_Frac_Self}, we know that there exists a self-similar measure $\nu_n \in \mathcal{M}(X)$ such that $\theta(\nu_n)=\nu_n$ and for any $\nu$, $\theta^k(\nu)\overset{k}{\to} \nu_n$ (in the $L$-metric), where \begin{equation*}
        \displaystyle{\theta^k(\nu)=\sum_{\sigma^1,\dots,~\sigma^k \in \lmbd{\om{}}{(n)}} (p_{\sigma^1}\dots p_{\sigma^k})\nu \circ (S_{\sigma^1}\circ \dots \circ S_{\sigma^k})^{-1}}.
        \end{equation*}
        Now following notations of Section \ref{sec:Invariant Measure}, we observe that for $j \in \N$, we have
        \begin{equation*}
            \mu_{\om{n}^p,\nu}^{(jn)}=\sum_{\sigma \in \lmbd{\om{n}^p}{(jn)}} p_{\sigma}(\nu \circ S_{\sigma}^{-1})
        \end{equation*}
        and for $\sigma \in \lmbd{\om{n}^p}{(jn)}$, we have 
        \begin{align*}
        p_{\sigma}=p_{\sigma^1}\dots p_{\sigma^j}~~\text{and}~~S_{\sigma}=S_{\sigma^1}\circ \dots \circ S_{\sigma^j},
        \end{align*}
    for some $\sigma^1,\dots,~\sigma^j \in \lmbd{\om{}}{n}$. This implies 
    \begin{equation*}
        \mu_{\om{n}^p,\nu}^{(jn)}=\sum_{\sigma^1,\dots,~\sigma^j \in \lmbd{\om{}}{n}} (p_{\sigma^1}\dots p_{\sigma^j}) \nu \circ (S_{\sigma^1}\circ \dots \circ S_{\sigma^j})^{-1}=\theta^j(\nu).
    \end{equation*}
   It follows that $\mu_{\om{n}^p,\nu}^{(jn)} \overset{j}{\to} \nu_n$. Hence by results in Section \ref{sec:Invariant Measure}, we have $\mu_{\om{n}^p}=\nu_n.$

    Now, if $\mathcal{I}$ satisfies the UESSC, then it is clear that the IFS $\mathbf{I}_{\om{},n}$ satisfies the strong separation condition (SSC), that is $S_{\sigma}(F_{\om{n}^p}),~\sigma \in \lmomgp{n}$, are disjoint, where $F_{\om{n}^p}$ is the attractor of the IFS $\mathbf{I}_{\om{},n}$. In that case, we know by \cite{Graf2000FoundationsOQ} that $D_r(\mu_{\om{n}^p})$ exists and it is given by 
    \begin{equation}\label{eqn:qnt dim of periodic}
        \pcr{\lmomgp{n}}{D_r(\mu_{\om{n}^p})}=1.
    \end{equation}
    For the convenience of notation, we will sometimes write $s_{n,r}=s_{n,r}(\om{})$ instead of $D_r(\mu_{\om{n}^p})$. We will use (\ref{eqn:qnt dim of periodic}) and other consequences of the periodic nature of $\om{n}^p$ in some of the upcoming results.

\section{Proofs}\label{sec:proofs}

\subsection{Proofs of necessary lemmas}\label{sec: proofs of lemmas}
In this section, we prove lemmas that will be used to prove our main theorem. Some of them are generalisations of lemmas proved by Zhu in \cite{ZHU2008_canor_like,ZHU_2012_MultiScaleMoran}. We begin by assuming that the RIFS $\mathcal{I}$ satisfies the UESSC for some $\beta >0$. 

Note that for $\sigma \in \lmomgp{n}$ the diameter of $E_{\sigma}$ is given by
\[\abs{E_{\sigma}}=(\prod_{j=1}^n c_{\om{j},\sigma_j}) \cdot \abs{X}.\] 
Without loss of generality, we assume that $\abs{X}=1$. So $\abs{E_{\sigma}}=c_{\sigma}$. Also, from (\ref{mu def sascha (n)}), one can deduce that \[\mu_{\om{}}(E_{\sigma})=(\prod_{j=1}^n p_{\om{j},\sigma_j})=p_{\sigma}.\]
Now, we define a subset $\gmomgn{n}$ of $\lmbd{\om{}}{*}$, which will play a crucial role in our theory. For $n \in \mathbb{N}$, define  
\begin{equation*}
    \gmomgn{n}:=\{\sigma \in \lmbd{\om{}}{*}: p_{\sigma^-}c_{\sigma^-}^r\geq n^{-1}(pc^r)>p_{\sigma}c_{\sigma}^r\} \subset \Lambda_{\omega}^*~,
\end{equation*} 
where $p:=\min\{p_{i,j}:j\in \mathbf{I}_i,~ i \in \lmbd{}{}\}$ and $c:=\min\{c_{i,j}:j\in \mathbf{I}_i,~ i \in \lmbd{}{}\}$. Note that $pc^r>0$.

Since $(p_{\sigma}c_{\sigma}^r) \to 0$ as $\abs{\sigma} \to \infty$, for $n \in \N$ there exist $\sigma \in \lmbd{\om{}}{*}$ such that $(p_{\sigma}c_{\sigma}^r)< pc^r/n \leq (p_{\sigma^-}c_{\sigma^-}^r)$. Hence for every $n$, $\Gamma_{\om{},n}$ is non-empty. Also, from the definition of $\gmomgn{n}$, it is evident that it is a finite maximal antichain. 

Now, let $t_{n,r}:=t_{n,r}(\omega)$ be the unique positive real number given by 
    \begin{equation}\label{eqn:tnr}
         \displaystyle{\sum_{\sigma \in \Gamma_{\om{},n}}\left(p_{\sigma}c_{\sigma}^r\right)^{\frac{t_{n,r}}{r+t_{n,r}}}=1},
    \end{equation}
    where existence and uniqueness of $t_{n,r}$ can be proved by arguments similar to Proposition $\ref{prop: existence of k_r}$. Also, from (\ref{eqn:tnr}) it follows that the sequence $\{t_{n,r}(\om{})\}_{n\geq1}$ is bounded, hence we define  \[\displaystyle{\underline{t}_r=\underline{t}_r(\om{}):=\liminf_{n \to \infty} t_{n,r}}~~~\text{and}~~~\displaystyle{\overline{t}_r=\overline{t}_r(\om{}):=\limsup_{n \to \infty} t_{n,r}}.\]
    Later (in Proposition \ref{lemma:Dr=tr}), we will see that if UESSC holds then $\underline{t_r}(\om{})=\underline{D}_r(\mu_{\om{}})$ and $\overline{t_r}(\om{})=\overline{D}_r(\mu_{\om{}})$.
    
    Also, set $l_{1n}=l_{1n}(\om{}):=\min_{\sigma \in \gmomgn{n}} \abs{\sigma},~l_{2n}=l_{2n}(\om{}):=\max_{\sigma \in \gmomgn{n}} \abs{\sigma}$, $\Phi_{\om{},n}:=$ $\operatorname{card}(\Gamma_{\om{},n})$ and for every $s>0$, set 
    \begin{align*}
    &\underline{P}_{r}^s(\mu_{\omega}):= \underset{n \to \infty}{\liminf}~ \Phi_{\om{},n}^{1/s} V_{\Phi_{\om{},n},r}^{1/r}(\mu_{\omega}),\\
    &\displaystyle{\overline{P}_{r}^s(\mu_{\omega}):= \underset{n \to \infty}{\limsup}~ \Phi_{\om{},n}^{1/s} V_{\Phi_{\om{},n},r}^{1/r}(\mu_{\omega})}.\end{align*}
For $\epsilon > 0$ and $A \subset \Rl^d$, define $(A)_{\epsilon}:=\{x \in \Rl^d: d(x,a)<\epsilon ~\text{for some}~ a \in A\}$, that is $(A)_{\epsilon}$ is the $\epsilon$-neighbourhood of the set $A$.

To estimate the upper quantization dimension, we need to construct efficient coverings of the random cylinder sets $E_\sigma$. The following lemma provides a uniform bound on the covering number of these sets, independent of the random realization $\omega$, which is essential for controlling the error across different scales.

Assume $D$ to be a constant such that $D^r > 2/ (pc^r).$

\begin{lemma}\label{lemma: G_1,G_2}

    There exist positive integers $G_1,G_2>1$ such that for any $\sigma \in \lmomgp{*}$, $E_{\sigma}$ can be covered by $G_1$ closed balls with radii $\beta\abs{E_{\sigma}}/(8D)$ and $(E_{\sigma})_{\beta\abs{E_{\sigma}}/4}$ can be covered by $G_2$ closed balls with radii $\beta\abs{E_{\sigma}}/(8D).$
\end{lemma}
\begin{proof}
    For $A \subset \Rl^d$, let $H_{A}$ be the largest number of mutually disjoint closed balls of radius $\beta\abs{A}/(16D)$ centred in $A$. Then calculating volumes of these balls corresponding to the set $E_{\sigma}$, we can deduce that
    \begin{equation*}
        H_{E_{\sigma}} \cdot \lr{(}{\beta\abs{E_{\sigma}}/(16D)}{)}^d \leq \lr{(}{\abs{E_{\sigma}} + \frac{\beta\abs{E_{\sigma}}}{16D}}{)}^d, 
    \end{equation*}
which gives $ H_{E_{\sigma}} \leq \lfloor \lr{(}{1+\frac{16D}{\beta}}{)}^d\rfloor$, where $\lfloor x \rfloor$ denotes the largest integer less than or equals to $x$.
Setting \[G_1:=\lfloor \lr{(}{1+\frac{16D}{\beta}}{)}^d\rfloor,\] we see that $G_1$ is independent of $\sigma$ and $E_{\sigma}$ can be covered by $G_1$ closed balls of radii $2 \cdot \beta\abs{A}/(16D)= \beta\abs{A}/(8D) $.

A similar calculation can be done for the set $(E_{\sigma})_{\beta\abs{E_{\sigma}}/4}$\\
resulting in $H_{(E_{\sigma})_{\beta\abs{E_{\sigma}}/4}} \leq \lfloor \lr{(}{1+\frac{16D}{\beta}+8D}{)}^d\rfloor$, which suggests \[G_2:=\lfloor \lr{(}{1+\frac{16D}{\beta}+8D}{)}^d\rfloor.\] This gives us the desired result.
\end{proof}

A key difficulty in quantization theory is ensuring that optimal points are distributed somewhat uniformly with respect to the measure. The next lemma establishes a 'finite local complexity' property: it asserts that the number of quantization centres falling within a specific neighbourhood of a cylinder set is uniformly bounded, preventing excessive clustering.

\begin{lemma}\label{lemma:Card(alpha_sigma)<L}
    There exists a constant $G \geq 1$ such that for any $m \leq \mathrm{card}(\Gamma_{\om{},n})$ and for any m-optimal set $\alpha$ for $V_{m,r}(\mu_{\omega})$, it holds for all $\sigma \in \Gamma_{\om{},n}$ that
    \begin{equation*}
        \mathrm{card}(\alpha_{\sigma}) \leq G,
    \end{equation*}
    where ${\alpha_{\sigma}}=\alpha \cap (E_{\sigma})_{\frac{\beta \abs{E_{\sigma}}}{8}}$.
\end{lemma}

    \begin{proof} Let $m \in \N$ be such that $m \leq \operatorname{card}(\gmomgn{n}{})$ and $\alpha$ be an arbitrary $m$-optimal set for $V_{m,r}(\mu_{\om{}})$.
    
    By UESSC, for any two distinct $\sigma, \tau \in \Gamma_{\om{},n}$, we have
    \[(E_{\sigma})_{\beta \abs{E_{\sigma}}/4} \cap (E_{\tau})_{\beta \abs{E_{\tau}}/4} = \emptyset.\]
Set  $G := G_1 + G_2$, as in Lemma \ref{lemma: G_1,G_2}.

If possible, suppose that there exist some $\sigma \in \Gamma_{\om{},n}$, such that $\operatorname{card}(\alpha_{\sigma}) > G$. Then for some $\tau \in \Gamma_{\om{},n} $, we have $\operatorname{card}(\alpha_{\tau}) = 0$,  since  $\operatorname{card}(\alpha) \leq \operatorname{card}(\Gamma_{\om{},n}).$

Let us choose $a_1, \ldots, a_G \in \alpha_{\sigma}$ (which is possible by the assumption $\operatorname{card}(\alpha_{\sigma}) > G$). Let $o_1, \ldots, o_{G_1} $ and  $f_1, \ldots, f_{G_2}$ be the centres of the $ G_1$ and $G_2$ closed balls with radii $\beta\abs{E_{\tau}}/(8D)$ and $\beta\abs{E_{\sigma}}/(8D)$ covering $ E_{\tau}$ and $ (E_{\sigma})_{{\beta}\abs{E_{\sigma}}/4}$, respectively (as in Lemma \ref{lemma: G_1,G_2}).

Setting $\alpha' := \left(\alpha \setminus \{a_1, \ldots, a_G\}\right) \cup \{o_1, \ldots, o_{G_1}, f_1, \ldots, f_{G_2}\}$, we get 
\begin{align*}
\int_{E_{\tau}} d(x,\alpha)^r d\mu_{\om{}}(x) \geq \frac{\beta^r\abs{E_{\tau}}^r}{8^r} \mu_{\om{}}(E_{\tau})&= \frac{\beta^r}{8^r}(p_{\tau} c_{\tau}^r) \\ &\geq \frac{(pc^r) \beta^r}{8^r}(p_{\tau^-} c_{\tau^-}^r)
\\ &\geq \frac{(pc^r) \beta^r}{8^r}((pc^r)/n) \\ &= \frac{(pc^r)^2\beta^r}{8^r n}
\end{align*}
and by the definition of $\alpha'$, we get 
\begin{align*}
&\int_{E_{\sigma}\cup E_{\tau}} d(x,\alpha')^r d\mu_{\om{}}(x) \\
&\leq \int_{E_{\sigma}} d(x,\alpha')^r d\mu_{\om{}}(x) + \int_{E_{\tau}} d(x,\alpha')^r d\mu_{\om{}}(x)\\
&\leq \frac{\beta^r\abs{E_{\sigma}}^r}{D^r 8^r} \mu_{\om{}}(E_{\sigma}) + \frac{\beta^r\abs{E_{\tau}}^r}{D^r 8^r} \mu_{\om{}}(E_{\tau})\\
&= \frac{\beta^r}{D^r 8^r} (p_{\sigma}c_{\sigma}^r+p_{\tau}c_{\tau}^r)\\
&< \frac{\beta^r}{D^r 8^r} (2(pc^r)/n)< \frac{(pc^r)^2\beta^r}{8^r n}.
\end{align*}
Hence, we have
\begin{equation}\label{inq: lemma card<L, 1}
    \int_{E_{\sigma}\cup E_{\tau}} d(x,\alpha)^r d\mu_{\om{}}(x)>\int_{E_{\sigma}\cup E_{\tau}} d(x,\alpha')^r d\mu_{\om{}}(x).
\end{equation}
For $y \in F_{\om{}} \setminus (E_{\sigma} \cup E_{\tau})$  and for any $ b \in \alpha_{\sigma}$, let $x_0$  be the intersection of the line between $y$ and $b$ and the surface of the closed ball $B(b, \beta\abs{E_{\sigma}}/8)$. Then $x_0 \in (E_{\sigma})_{\beta\abs{E_{\sigma}}/4}.$
 
Since \[\displaystyle{(E_{\sigma})_{\beta \abs{E_{\sigma}}/4} \subset \bigcup_{k=1}^{G_2} B(f_k,\beta \abs{E_{\sigma}}/8D)}~,\]
 there exits $1\leq k \leq G_2$ such that $x_0 \in B(f_k,\beta \abs{E_{\sigma}}/8D)$. So,
\begin{align*}
\|y - f_k\| &\leq \|y - x_0\| + \|x_0 - f_k\| \\ 
&\leq \|y - x_0\| + \frac{\beta\abs{E_{\sigma}}}{8D} \\
&\leq \|y - x_0\| + \frac{\beta\abs{E_{\sigma}}}{8}=\|y - b\|,
\end{align*}
yielding $\displaystyle{ d(y,\alpha_{\sigma}) \geq  \min_{1 \leq k \leq G_2} \|y - f_k\|}
$. Then it follows that for all $y \in F_{\om{}} \setminus (E_{\sigma}\cup E_{\tau})$
\begin{equation*}
    d(y,\alpha)\geq d(y,\alpha')
\end{equation*}
and hence 
\begin{equation}\label{inq: lemma card<L, 2}
     \int_{F_{\om{}} \setminus (E_{\sigma}\cup E_{\tau})} d(y,\alpha)^r d\mu_{\om{}}(y)>\int_{F_{\om{}} \setminus (E_{\sigma}\cup E_{\tau})} d(y,\alpha')^r d\mu_{\om{}}(y).
\end{equation}
Using (\ref{inq: lemma card<L, 1}) and (\ref{inq: lemma card<L, 2}), we get

\begin{align*}
V_{m,r}(\mu_{\om{}}) &= \sum_{\sigma \in \Gamma_{\om{},n}} \int_{E_{\sigma}} d(x,\alpha)^r d\mu_{\om{}}(x) \\ &> \sum_{\sigma \in \Gamma_{\om{},n}} \int_{E_{\sigma}} d(x,\alpha')^r d\mu_{\om{}}(x)\\
&=  \int_{F_{\om{}}}  d(x,\alpha')^r d\mu_{\om{}}(x)\geq V_{m,r}(\mu_{\om{}}).
\end{align*}
This is a contradiction since $\operatorname{card}(\alpha')=\operatorname{card}(\alpha)\leq m$ and $\alpha$ is a $m$-optimal set. Hence, the lemma follows. 
\end{proof}

Now, we also ensure that the quantization error does not vanish too quickly on any cylinder. The following lemma establishes a uniform lower bound proportional to the geometric size of the cylinder.

\begin{lemma}\label{lemma:int d(x,alpha)>Dh(sigma)}
    Let $\alpha$ be a non-empty finite subset of $\mathbb{R}^d$. Then there exists a constant $M=M(\alpha)>0$ such that
    \begin{equation*}
        \underset{E_{\sigma}}{\int} d(x,\alpha)^r d \mu_{\omega}(x) \geq M \cdot (p_{\sigma}c_{\sigma}^r)    
    \end{equation*}
    holds for any $\sigma \in \lmbd{\om{}}{*}$.
\end{lemma}

\begin{proof}
Let $\sigma \in \lmbd{\om{}}{*}$ and $J \geq 1$ be such that $2^J > \operatorname{card}(\alpha)$.

If possible, suppose for some $b \in \alpha$ there exists \\
$\tau_1 \neq \tau_2 \in \lmbd{J}{}(\sigma)$ such that
\begin{equation}\label{inq: lemma:int d(x,alpha)>Dh(sigma), 1 }
    \operatorname{dist}(b, E_{\tau_i}) < \frac{\beta}{2} \min\{\abs{E_{\tau}}: \tau \in \Lambda_J(\sigma)\}, \quad  i= 1, 2,
\end{equation}
where $\operatorname{dist}(b, E_{\tau_i})=\operatorname{dist}(\{b\}, E_{\tau_i})$. Then we can deduce 
\[\operatorname{dist}(E_{\tau_1}, E_{\tau_2}) < \beta \min\{\abs{E_{\tau}}: \tau \in \Lambda_J(\sigma)\},\]
which contradicts UESSC. So for each $ b \in \alpha$ there is at most one $\tau \in \Lambda_J(\sigma)$ such that (\ref{inq: lemma:int d(x,alpha)>Dh(sigma), 1 }) holds. Again, since
$\operatorname{card}(\alpha) <  2^J \leq \operatorname{card}(\lmbd{J}{}(\sigma))$, there exists some $\tau' \in \lmbd{J}{}(\sigma)$ such that
\begin{equation}\label{inq:min dist(a,E_T)}
 \min_{b \in \alpha} \operatorname{dist}(b, E_{\tau'}) \geq \frac{\beta}{2} \min\{\abs{E_{\tau}}: \tau \in \Lambda_J(\sigma)\}.
    \end{equation}
Utilizing (\ref{inq:min dist(a,E_T)}), we get
\begin{align*}
&\int_{E_{\sigma}} d(x,\alpha)^r d\mu_{\om{}}(x)\\
&\geq \int_{E_{\tau'}} d(x,\alpha)^r d\mu_{\om{}}(x)\\
&\geq \mu_{\om{}}(E_{\tau'}) \frac{\beta^r}{2^r} \lr{[}{\min\{\abs{E_{\tau}}: \tau \in \Lambda_J(\sigma)\}}{]}^r \\
&\geq (p^Jp_{\sigma}) \frac{\beta^r}{2^r} (c^J\abs{E_{\sigma}})^r= (pc^r)^J  \frac{\beta^r}{2^r} (p_{\sigma}c_{\sigma}^r) =: M(p_{\sigma}c_{\sigma}^r)~,
\end{align*}
where $M = (pc^r)^J  \frac{\beta^r}{2^r}.$ 
This completes the proof.
\end{proof}

We now turn to the lower bound for the quantization error. By utilizing the separation condition (UESSC) and the mass distribution principle, we derive the following estimate, which relates the $n$-th quantization error to the geometric scale of the maximal antichain $\Gamma_{\omega,n}$.
 
\begin{lemma}\label{lemma:D<phi V_phi}
    There exists a positive constant $\widetilde{D}$ such that for $n \in \N$ 
    \begin{equation*}
        V_{\Phi_{\om{},n},r}(\mu_{\omega}) > \widetilde{D}~ \Phi_{\om{},n}^{-r/t_{n,r}}. 
    \end{equation*}
\end{lemma}
\begin{proof} For $n \in \N$, let $\alpha$ be a $\Phi_{\om{},n}$-optimal set for $V_{\Phi_{\om{},n},r}(\mu_{\om{}})$. For $\sigma \in \gmomgn{n}$, let $q_1,\dots,q_{G_1}$ be the centres of $G_1$ closed balls with radii $\beta \abs{E_{\sigma}}/(8D)$, which covers $E_{\sigma}$ (as in Lemma \ref{lemma: G_1,G_2}). 
Also, set $\overline{\alpha}_{\sigma}:=\alpha_{\sigma}\cup \{q_1,\dots,q_{G_1}\}$, where $\alpha_{\sigma}=\alpha \cap (E_{\sigma})_{\beta \abs{E_{\sigma}}/8}.$ Then for $x \in E_{\sigma}$, we have 
\begin{equation}\label{inq:d(x,alpha)>d(x,alpha^~)}
    d(x,\alpha)\geq d(x,\overline{\alpha}_{\sigma}).
\end{equation}
By Lemma \ref{lemma:Card(alpha_sigma)<L}, we have
 $\operatorname{card}(\overline{\alpha}_{\sigma})\leq G + G_1=:\widetilde{G}.$
 
Then we get
\begin{align}\label{inq:V_{phi_n}Dsumh(sigma)}
    \nonumber V_{\Phi_{\om{},n},r}(\mu_{\om{}})&=\sum_{\sigma \in \gmomgn{n}} \int_{E_{\sigma}} d(x,\alpha)^r d\mu_{\om{}}(x)\\
     \nonumber &\geq \sum_{\sigma \in \gmomgn{n}} \int_{E_{\sigma}} d(x,\overline{\alpha}_{\sigma})^r d\mu_{\om{}}(x)~~~~(\text{by (\ref{inq:d(x,alpha)>d(x,alpha^~)})})\\
      &\geq \sum_{\sigma \in \gmomgn{n}} M (p_{\sigma}c_{\sigma}^r)~~~~(\text{by Lemma \ref{lemma:int d(x,alpha)>Dh(sigma)}}),
 \end{align}
 % \begin{equation}
 %     \implies V_{\Phi_{\om{},n},r}(\mu_{\om{}}) \geq M \sum_{\sigma \in \gmomgn{n}} (p_{\sigma}c_{\sigma}^r),
 % \end{equation}
     where $M = (pc^r)^J  \frac{\beta^r}{2^r}$ and $J$ is given by, $2^J>\widetilde{G}\geq\operatorname{card}(\overline{\alpha}_{\sigma})$. So $M$ is independent of $n$.\\
Note that, for
    $\sigma \in \gmomgn{n}$~, using (\ref{eqn:tnr}), we have  \[\Phi_{\om{},n}^{-r/t_{n,r}} < (pc^r/n)^{r/(r+t_{n,r})} \leq (p_{\sigma^-}c_{\sigma^-}^r)^{r/(r+t_{n,r})}
    \] from which, we can deduce
    \[\Phi_{\om{},n}^{-r/t_{n,r}} (pc^r)^{r/(r+t_{n,r})} < (p_{\sigma}c_{\sigma}^r)^{r/(r+t_{n,r})}.\]
Then it follows that
\begin{align}\label{lemmaa 4.4 (last)}
    \nonumber (p_{\sigma}c_{\sigma}^r)&>(pc^r)^{r/(r+t_{n,r})} (p_{\sigma}c_{\sigma}^r)^{t_{n,r}/(r+t_{n,r})} \Phi_{\om{},n}^{-r/t_{n,r}}
    \\ 
    &> (pc^r) (p_{\sigma}c_{\sigma}^r)^{t_{n,r}/(r+t_{n,r})} \Phi_{\om{},n}^{-r/t_{n,r}}.
    \end{align} 
    Hence \[\displaystyle{\sum_{\sigma \in \gmomgn{n}} (p_{\sigma}c_{\sigma}^r) > (pc^r)\Phi_{\om{},n}^{-r/t_{n,r}}}.\]
Using this in (\ref{inq:V_{phi_n}Dsumh(sigma)}), we get
\begin{equation*}
    V_{\Phi_{\om{},n},r}(\mu_{\om{}}) \geq M \sum_{\sigma \in \gmomgn{n}} (p_{\sigma}c_{\sigma}^r) > \widetilde{D}~ \Phi_{\om{},n}^{-r/t_{n,r}},
\end{equation*}
where $\widetilde{D}=M(pc^r)$ is a positive constant independent of $n$.
\end{proof}

The following lemma provides an upper bound for the quantization error.

\begin{lemma}\label{lemma:phi V_phi<D}
    There exists a positive constant $\underset{\widetilde{}}{D}$ such that for large enough $n \in \N$ 
    \begin{equation*}
      V_{\Phi_{\om{},n},r}(\mu_{\omega}) \leq \underset{\widetilde{}}{D}~\Phi_{\om{},n}^{-r/t_{n,r}},
    \end{equation*}
    where $\underset{\widetilde{}}{D}$ is independent of $n$. 
\end{lemma}
\begin{proof}
For $\sigma \in \gmomgn{n}$, let $x_{\sigma} \in E_{\sigma}$ be arbitrary.\\

\noindent Set $\alpha_0:=\{x_{\sigma}: \sigma \in \gmomgn{n}\}.$ Then 
% \newpage
\begin{align*}
    &V_{\Phi_{\om{},n},r}(\mu_{\om{}}) \\
    &\leq \sum_{\sigma \in \gmomgn{n}} \int_{E_{\sigma}} d(x,\alpha_0)^r d\mu_{\om{}}(x)\\
    &\leq \sum_{\sigma \in \gmomgn{n}} \abs{E_{\sigma}}^r \mu_{\om{}}(E_{\sigma})
    = \sum_{\sigma \in \gmomgn{n}} c_{\sigma}^r p_{\sigma}\\
    %&= \sum_{\sigma \in \gmomgn{n}} (p_{\sigma}c_{\sigma}^r)^{\frac{t_{n,r}}{r+t_{n,r}}} (p_{\sigma}c_{\sigma}^r)^{\frac{r}{r+t_{n,r}}}\\
    &\leq \sum_{\sigma \in \gmomgn{n}} (p_{\sigma}c_{\sigma}^r)^{\frac{t_{n,r}}{r+t_{n,r}}} \left(n^{-1}(pc^r)\right)^{\frac{r}{r+t_{n,r}}}\\
    %&=\left(n^{-1}(pc^r)\right)^{\frac{r}{r+t_{n,r}}}\\
    &= \left[\left(n^{-1}(pc^r)\right)^{\frac{t_{n,r}}{r+t_{n,r}}} \cdot(pc^r)^{\frac{t_{n,r}}{r+t_{n,r}}} \cdot \Phi_{\om{},n} \right]^{\frac{r}{t_{n,r}}} (pc^r)^{\frac{-r}{r+t_{n,r}}} \Phi_{\om{},n}^{\frac{-r}{t_{n,r}}}\\
    &\leq \lr{[}{\sum_{\sigma \in \gmomgn{n}} (p_{\sigma}c_{\sigma}^r)^{\frac{t_{n,r}}{r+t_{n,r}}}   }{]}^{\frac{r}{t_{n,r}}} (pc^r)^{\frac{-r}{r+t_{n,r}}} \Phi_{\om{},n}^{\frac{-r}{t_{n,r}}}~~~~(\text{by definition of}~ \gmomgn{n})\\
    &= (pc^r)^{\frac{-r}{r+t_{n,r}}} \Phi_{\om{},n}^{\frac{-r}{t_{n,r}}}~.
\end{align*}
By the definition of $\underline{t}_r$, for large enough $n \in \N$, we have\\
$\frac{1}{2}\underline{t}_r<t_{n,r}$, which gives $\frac{r}{r+t_{n,r}}< \frac{2r}{2r+\underline{t}_r}$ and hence 
\begin{align*}
     V_{\Phi_{\om{},n},r}(\mu_{\om{}}) \leq  (pc^r)^{\frac{-r}{r+t_{n,r}}} \Phi_{\om{},n}^{\frac{-r}{t_{n,r}}}
     \leq  (pc^r)^{\frac{-2r}{2r+\underline{t}_r}} \Phi_{\om{},n}^{\frac{-r}{t_{n,r}}}=\underset{\widetilde{}}{D}~ \Phi_{\om{},n}^{\frac{-r}{t_{n,r}}}~,
\end{align*}
where $\underset{\widetilde{}}{D} = (pc^r)^{\frac{-2r}{2r+\underline{t}_r}}$. Hence the proof.
\end{proof}

The dimension calculation relies on comparing the quantization error to the size of the covering antichain $\Gamma_{\omega,n}$. To facilitate this, we first establish the growth properties of the sequence $\Phi_{\omega,n}$.

\begin{lemma}\label{lemma:phi_ngoes to infinity}
    For large enough $n \in \N$, the following holds:
    \begin{equation*}
        \Phi_{\om{},n}\leq \Phi_{\om{},n+1} \leq \widetilde{N} \Phi_{\om{},n},
    \end{equation*}
    where $\widetilde{N}:=\max\{\mathrm{card}(\mathbf{I}_i):1 \leq i \leq N\}.$
    Furthermore, \\
    $\Phi_{\om{},n} \to \infty ~\text{as}~ n \to \infty.$
\end{lemma}
\begin{proof}
For $n \in \N$, let $\sigma \in \gmomgn{n}$. Then 
\begin{equation}\label{inq: Phi_n<Phi_n+1 0}
p_{\sigma}c_{\sigma}^r<(pc^r)/n\leq p_{\sigma^-}c_{\sigma^-}^r.
\end{equation}
Set ${\widetilde{p}:=\max_{i,j}\{p_{i,j}\}}$ and ${\widetilde{c}:=\max_{i,j}\{c_{i,j}\}}$. Then we have \\
$0<(\widetilde{p}~\tilde{c}^r)<1$.

Also
$((\widetilde{p}~\tilde{c}^r)/n)<1/(n+1)$ if and only if  $n>((\widetilde{p}~\widetilde{c}^r)^{-1}-1)^{-1}.$ So for $\tau \in \Lambda_1(\sigma)$, we can deduce
\begin{equation}\label{inq:Phi_n<Phi_(n+1)}
(p_{\tau}c_{\tau}^r) \leq (p_{\sigma}c_{\sigma}^r)(\widetilde{p}~\tilde{c}^r)<(pc^r/n)(\widetilde{p}~\tilde{c}^r)<pc^r/(n+1),
\end{equation}
when $n>((\widetilde{p}~\widetilde{c}^r)^{-1}-1)^{-1}.$

Note that if $\sigma \notin \gmomgn{(n+1)}$ then from $(\ref{inq: Phi_n<Phi_n+1 0})$ it follows that $(pc^r)/(n+1) \leq p_{\sigma}c_{\sigma}^r$. Hence by $(\ref{inq:Phi_n<Phi_(n+1)})$ for $n>((\widetilde{p}~\widetilde{c}^r)^{-1}-1)^{-1}$, we have
\begin{align*}
p_{\tau}c_{\tau}^r<(pc^r)/(n+1)\leq p_{\sigma}c_{\sigma}^r=p_{\tau^-}c_{\tau^-}.
\end{align*}
That is $\tau \in \gmomgn{(n+1)}$. So either $\sigma \in \gmomgn{(n+1)}$ or $\tau \in \gmomgn{(n+1)}$ . Hence for $n > \lr{(}{(\Tilde{p}\Tilde{c}^r)^{-1}~-1}{)}^{-1}$, we have \begin{equation}\label{inq: Phi_n<Phi_n+1 Final}
\Phi_{\om{},n}\leq \Phi_{\om{},n+1} \leq \widetilde{N} \Phi_{\om{},n}.
\end{equation}
Let $n_0 \geq \lr{(}{(\Tilde{p}\Tilde{c}^r)^{-1}~-1}{)}^{-1}$. Since $pc^r/n \to 0$ as $n \to \infty$ there exists $n_1>n_0$ such that $\underset{\widetilde{}}{N} \Phi_{\om{},n_0} \leq \Phi_{\om{},n_1},$ where ${\underset{\widetilde{}}{N}:=\min_{1\leq i \leq N}\{\mathrm{card}(\mathbf{I}_i)\} \geq 2}.$ 

Continuing this process, we get $\{n_j\}_{j\geq 0} \subset \N$ such that $\underset{\widetilde{}}{N}^j \Phi_{\om{},n_0} \leq \Phi_{\om{},n_j},~j\geq 1,$ which implies $\Phi_{\om{},n_j} \to \infty$ as $j \to \infty.$ Hence by $(\ref{inq: Phi_n<Phi_n+1 Final})$, we have
\begin{equation*}
    \Phi_{\om{},n} \to \infty ~\text{as}~ n \to \infty.
\end{equation*}
Hence, the proof.
\end{proof}

\begin{corollary}\label{cor:phi_j<n<phi_j+1}
    For every $n\geq \Phi_{\om{},1}$, there exists $j \in \N$ such that $\Phi_{\om{},j} \leq n < \Phi_{\om{},j+1}.$
\end{corollary}
\begin{proof}
    This follows from Lemma \ref{lemma:phi_ngoes to infinity}.
\end{proof}

\begin{lemma}\label{lemma:P>Q>NQ and  P<Q<NQ}
       $\underline{P}_{r}^s(\mu_{\omega}) \geq \underline{Q}_{r}^s(\mu_{\omega}) \geq \widetilde{N}^{-1/s}~\underline{P}_{r}^s(\mu_{\omega})$
    and $\overline{P}_{r}^s(\mu_{\omega}) \leq \overline{Q}_{r}^s(\mu_{\omega}) \leq \widetilde{N}^{1/s}~\overline{P}_{r}^s(\mu_{\omega})$, holds for all $s>0$.
\end{lemma}
\begin{proof}
   It follows from Lemma \ref{lemma:phi_ngoes to infinity} and Corollary \ref{cor:phi_j<n<phi_j+1} that
    \begin{align*}
        \underline{P}_{r}^s(\mu_{\omega}) \geq \underline{Q}_{r}^s(\mu_{\omega})  &\geq \liminf_{j \to \infty}~ \Phi_{\om{},j}^{1/s} V_{\Phi_{\om{},j+1},r}^{1/r}(\mu_{\om{}})\\
         &\geq \widetilde{N}^{-1/s} \liminf_{j \to \infty}~ \Phi_{\om{},j+1}^{1/s} V_{\Phi_{\om{},j+1},r}^{1/r}(\mu_{\om{}})\\
         &= \widetilde{N}^{-1/s}  \underline{P}_{r}^s(\mu_{\omega}) .
    \end{align*}
    Likewise, inequalities for the limit supremum follow.
\end{proof}

\begin{corollary}\label{corr:P iff Q}
    $\underline{P}_{r}^s(\mu_{\omega})>0 \iff \underline{Q}_{r}^s(\mu_{\omega})>0$ ~and~ $\overline{P}_{r}^s(\mu_{\omega})<\infty \iff \overline{Q}_{r}^s(\mu_{\omega})<\infty$. 
\end{corollary}
\begin{proof}
    These are directly derived from Lemma \ref{lemma:P>Q>NQ and  P<Q<NQ}.
\end{proof}

The auxiliary parameters $\underline{t}_r(\omega),~\bar{t}_r(\omega)$, derived from the set $\Gamma_{\omega,n}$, serve as a bridge between the geometric scaling of the system and the quantization error. In the following proposition, we rigorously identify these parameters with the lower and upper quantization dimensions.

\begin{proposition}\label{lemma:Dr=tr}
    $\underline{D}_r(\mu_{\omega})=\underline{t}_r(\om{})$ and $\overline{D}_r(\mu_{\omega})=\overline{t}_r(\om{})$.
\end{proposition}
\begin{proof}
If possible, let $0\leq\underline{t}_r<\underline{D}_r(\mu_{\omega})$. Then there exists a large $n \in \N$ such that $0 \leq  t_{n,r}<\underline{D}_r(\mu_{\omega})$. Then by \cite[Proposition 11.3]{Graf2000FoundationsOQ}, we have $\overline{Q}_r^{t_{n,r}}(\mu_{\omega})=\infty$. Hence, by Lemma \ref{lemma:P>Q>NQ and  P<Q<NQ}, it follows that $\overline{P}_r^{t_{n,r}}(\mu_{\omega})=\infty$. This contradicts Lemma \ref{lemma:phi V_phi<D}.

Again, if we assume that $\underline{t}_r>\underline{D}_r(\mu_{\omega})$ then for some large $n \in \N$, we have $t_{n,r}>\underline{D}_r(\mu_{\omega})$. Then by \cite[Proposition 11.3]{Graf2000FoundationsOQ}, we can deduce $\underline{Q}_r^{~t_{n,r}}(\mu_{\omega})=0$, which implies (by Corollary \ref{corr:P iff Q}) $\underline{P}_r^{~t_{n,r}}(\mu_{\omega})=0$. This contradicts Lemma \ref{lemma:D<phi V_phi}.

Hence $\underline{D}_r(\mu_{\omega})=\underline{t}_r$. Similarly, the other equality can be proved
\end{proof}

\begin{lemma}\label{lemma: ljn goes to infinity}
    For $j=1,2$, $l_{jn} (\om{})\to \infty$ as $n \to \infty$.
\end{lemma}

\begin{proof}
Suppose that for all $n \in \N$, $\displaystyle{1\leq l_{1n}\leq l_0}$
for some $l_0\in \N$. Then for every $n \in \N$ there exists $\sigma^{(n)} \in \gmomgn{n}$ with $\abs{\sigma^{(n)}}=l_{1n}$ so that \begin{equation}\label{inq:ljn go to infty 1}
(p_{\sigma^{(n)}}c_{\sigma^{(n)}}^r)\geq (pc^r)^{l_{1n}}\geq (pc^r)^{l_0}.
\end{equation}
Also, for large enough $n \in \N$, we have \begin{equation}\label{inq:ljn go to infty 2}
(pc^r/n)<(pc^r)^{l_0}.
\end{equation}
Combining (\ref{inq:ljn go to infty 1}) and (\ref{inq:ljn go to infty 2}), we have for large enough $n \in \N$
\[(p_{\sigma^{(n)}}c_{\sigma^{(n)}}^r)<(pc^r/n)<(pc^r)^{l_0}\leq (p_{\sigma^{(n)}}c_{\sigma^{(n)}}^r),\] which is not possible. 

Since from the proof of Lemma $\ref{lemma:phi_ngoes to infinity}$, it is evident that for large enough $n\in \N$,  $l_{1(n+1)}\geq l_{1n}$ and $l_{2(n+1)}\geq l_{2n}$, we deduce that $l_{2n}\geq l_{1n} \to \infty$ as $n \to \infty.$ 
\end{proof}

\begin{lemma}\label{lemma:l_n^1 l_n^2}
    For every $n \in \mathbb{N}$ there exists $l_n^{(1)}=l_n^{(1)}(\om{}),~l_n^{(2)}=l_n^{(2)}(\om{}) \in \mathbb{N}$ such that $l_{1n}(\om{}) \leq l_n^{(1)}(\om{}), l_n^{(2)}(\om{}) \leq l_{2n}(\om{})$ and $s_{l_n^{(1)},r}(\om{})\leq t_{n,r}(\om{}) \leq s_{l_n^{(2)},r}(\om{})$.
\end{lemma}

\begin{proof}
For $n \in \N$, set \[{\underline{s}_{n,r}:=\min_{l_{1n}\leq k \leq l_{2n}} s_{k,r}},~{\overline{s}_{n,r}:=\max_{l_{1n}\leq k \leq l_{2n}} s_{k,r}}.\]

For $l_{1n}\leq k \leq l_{2n},$ define $T_k:=\pcr{\lmomgp{k}}{\overline{s}_{n,r}}$. Then \[T_k\leq \pcr{\lmomgp{k}}{s_{k,r}}=1.\]
For $l_{1n}\leq k \leq l_{2n}$ and for $\sigma \in \lmomgp{k}$ define $\Xi(\sigma):=T_{\abs{\sigma}}^{-1}(p_{\sigma}c_{\sigma}^r)^{\frac{\overline{s}_{n,r}}{r+\overline{s}_{n,r}}}=T_{k}^{-1}(p_{\sigma}c_{\sigma}^r)^{\frac{\overline{s}_{n,r}}{r+\overline{s}_{n,r}}}.$\\

Then for $j \in \mathbb{N}$, we have
\begin{align*}
    &\sum_{\tau \in \lmbd{j}{}(\sigma)} \Xi(\tau)\\
    &=T_{\abs{\tau}}^{-1} \sum_{\tau \in \lmbd{j}{}(\sigma)} (p_{\tau}c_{\tau}^r)^{\frac{\overline{s}_{n,r}}{r+\overline{s}_{n,r}}}\\ 
    &= T_{k+j}^{-1} \cdot (p_{\sigma}c_{\sigma}^r)^{\frac{\overline{s}_{n,r}}{r+\overline{s}_{n,r}}} \cdot \prod_{i=1}^j \lr{[}{\sum_{\tau_{k+i} \in \mathbf{I}_{\om{k+i}}}\lr{(}{p_{\om{k+i},\tau_{k+i}}c_{\om{k+i},\tau_{k+i}}^r}{)}^{\frac{\overline{s}_{n,r}}{r+\overline{s}_{n,r}}}}{]}\\
    &= T_{k+j}^{-1} \cdot (p_{\sigma}c_{\sigma}^r)^{\frac{\overline{s}_{n,r}}{r+\overline{s}_{n,r}}} \cdot \prod_{i=1}^j\lr{(}{\frac{T_{k+i}}{T_{k+i-1}}}{)}\\
    &= T_{k}^{-1} \cdot (p_{\sigma}c_{\sigma}^r)^{\frac{\overline{s}_{n,r}}{r+\overline{s}_{n,r}}}=\Xi(\sigma).
\end{align*}
Using the fact that $\Gamma_{\om{},n}$ is a FMA, for $l_{1n}\leq k \leq l_{2n}$, we deduce  
\begin{align*}
    &\sum_{\sigma \in \gmomgn{n}} \Xi (\sigma)=\sum_{\sigma \in \lmomgp{k}} \Xi (\sigma) \\
\implies & \sum_{\sigma \in \gmomgn{n}} T_{\abs{\sigma}}^{-1}(p_{\sigma}c_{\sigma}^r)^{\frac{\overline{s}_{n,r}}{r+\overline{s}_{n,r}}}=1.
\end{align*} 
Since $T_{\abs{\sigma}}\leq 1$, it follows that 
\[\sum_{\sigma \in \gmomgn{n}}(p_{\sigma}c_{\sigma}^r)^{\frac{\overline{s}_{n,r}}{r+\overline{s}_{n,r}}}\leq \sum_{\sigma \in \gmomgn{n}} T_{\abs{\sigma}}^{-1}(p_{\sigma}c_{\sigma}^r)^{\frac{\overline{s}_{n,r}}{r+\overline{s}_{n,r}}} = 1 \implies t_{n,r}\leq \overline{s}_{n,r}.\]
Similarly, it can be shown that $t_{n,r}\geq \underline{s}_{n,r}$. Hence the proof.
\end{proof}

\subsection{Proofs concerning almost sure quantization dimension}\label{subsec:exact quant.}

The transition from deterministic to random dynamics requires us to control the asymptotic behaviour of the dimension function along typical realisations. By invoking the \textit{Strong Law of Large Numbers} (SLLN) on the symbolic space $\Omega$, we now show that the quantization dimension of the periodic approximations $\mu_{\omega_n^p}$ stabilizes to the constant $\kappa_r$ for $\mathbf{P}$-almost all realizations $\omega$.

\begin{proposition}
    \label{lemma:Strng law of large numbers}
    For almost all $\omega \in \Omega$ 
    
    \begin{equation*}
  \lim_{n \to \infty} D_{r}(\mu_{\omega_{n}^p})= \kappa_{r}.
    \end{equation*}
    \end{proposition}

\begin{proof}
    Recall that for a fixed $n$ and realization $\omega$, the $n$-th quantization dimension $s_{n,r}(\omega) := D_r(\mu_{\om{n}^p})$ is the unique root of the equation:
    \begin{equation*}\label{eq:pressure_root_n}
        \sum_{\sigma \in \Lambda_{\omega}^{(n)}} (p_{\sigma} c_{\sigma}^r)^{\frac{s_{n,r}(\omega)}{r + s_{n,r}(\omega)}} = 1.
    \end{equation*}
    For $s>0$, we define the \textit{empirical pressure function} $\Psi_{\om{},n}(s)$ by taking the normalized logarithm of the sum:
    \begin{equation*}\label{eq:empirical_pressure}
       \Psi_{\om{},n}(s) := \frac{1}{n} \log \sum_{\sigma \in \Lambda_{\omega}^{(n)}} (p_{\sigma} c_{\sigma}^r)^{\frac{s}{r + s}}.
    \end{equation*}
    By definition, $\Psi_{\om{},n}(s_{n,r}(\omega)) = 0$ for all $n \ge 1$.
    
    Due to the multiplicative structure of the weights in the RIFS, the term inside the logarithm factorizes. Specifically, for any $\sigma = (\sigma_1, \dots, \sigma_n)$, the weight decomposes as $p_\sigma c_\sigma^r = \prod_{i=1}^n p_{\omega_i, \sigma_{i}} c_{\omega_i, \sigma_{i}}^r$. Consequently, $\Psi_n$ can be expressed as a Birkhoff average of independent random variables:
    \begin{equation*}
        \Psi_{\om{},n}(s) = \frac{1}{n} \sum_{i=1}^n X_i(s, \omega),
    \end{equation*}
    where $X_{\om{i}}(s) := \log \left( \sum_{j \in I_{\omega_i}} (p_{\omega_i, j} c_{\omega_i, j}^r)^{\frac{s}{r+s}} \right)$. Since the indices $\omega_i$ are chosen independently according to the measure $\mathbf{P}$, the sequence $\{X_{\om{i}}(s)\}_{i \ge 1}$ consists of independent and identically distributed (i.i.d.) bounded random variables.
    
    We define the \textit{expected pressure function} $\Psi(s)$ as the expectation of a single increment:
    \begin{equation*}
        \Psi(s) := \mathbb{E}[X_1(s)] = \sum_{k=1}^N \zeta_k \log \left( \sum_{j \in I_k} (p_{k, j} c_{k, j}^r)^{\frac{s}{r+s}} \right).
    \end{equation*}
    From Proposition \ref{prop: existence of k_r}, the quantization dimension $\kappa_r$ is defined as the unique zero of this function, i.e., $\Psi(\kappa_r) = 0$. Furthermore, observing that the base terms $p_{k,j} c_{k,j}^r$ are strictly less than 1 and the exponent $s \mapsto \frac{s}{r+s}$ is strictly increasing for $s>0$, the function $\Psi(s)$ is strictly decreasing and continuous on $(0, \infty)$.
    
    Now, fix a bounded closed interval $J = [\kappa_r - \epsilon, \kappa_r + \epsilon]$ with $\epsilon > 0$ such that $\kr-\epsilon>0$. For any fixed $s \in J$, SLLN implies that:
    \begin{equation*}
       \Psi_{\om{},n}(s)\xrightarrow{n \to \infty} \Psi(s) \quad \text{for } \mathbf{P}\text{-almost all } \omega.
    \end{equation*}
    Since both  $\Psi_{\om{},n}(s)$ and $\Psi(s)$ are continuous and strictly monotonic, the pointwise convergence implies uniform convergence on the compact interval $J$ (a consequence of Dini's Theorem). Thus, for $\mathbf{P}$-almost all $\omega$,
    \begin{equation}\label{eq:uniform_conv}
        \sup_{s \in J} |\Psi_{\om{},n}(s) - \Psi(s)| \xrightarrow{n \to \infty} 0.
    \end{equation}
    Finally, we deduce the convergence of the roots $s_{n,r}(\omega)$. Since $\Psi(s)$ is strictly decreasing and $\Psi(\kappa_r)=0$, we have:
    $$ \Psi(\kappa_r - \epsilon) > 0 \quad \text{and} \quad \Psi(\kappa_r + \epsilon) < 0. $$
    By the almost sure uniform convergence \eqref{eq:uniform_conv}, there exists a random integer $N(\omega)$ such that for all $n \ge N(\omega)$,
    $$ |\Psi_{\om{},n}(\kappa_r \pm \epsilon) - \Psi(\kappa_r \pm \epsilon)| < \frac{1}{2} \min \{ |\Psi(\kappa_r - \epsilon)|, |\Psi(\kappa_r + \epsilon)| \}. $$
    This estimate ensures that the empirical pressure preserves the signs at the endpoints:
    $$ \Psi_{\om{},n}(\kappa_r - \epsilon) > 0 \quad \text{and} \quad \Psi_{\om{},n}(\kappa_r + \epsilon) < 0. $$
    Since $s \mapsto \Psi_{\om{},n}(s)$ is continuous, the Intermediate Value Theorem implies that its unique root $s_{n,r}(\omega)$ must satisfy:
    $$ \kappa_r - \epsilon < s_{n,r}(\omega) < \kappa_r + \epsilon. $$
    Since $\epsilon$ was arbitrary, we conclude that $\lim_{n \to \infty} s_{n,r}(\omega) = \kappa_r$ almost surely.
\end{proof}

 \begin{proposition}\label{prop: D_r=k_r}
     For almost all $\om{} \in \Om{}$ 
     \begin{equation*}
         D_r(\mu_{\om{}})=\kr.
     \end{equation*}
      \end{proposition}
     \begin{proof}
By Lemmas \ref{lemma: ljn goes to infinity} and \ref{lemma:l_n^1 l_n^2}, we deduce that for all $\om{} \in \Om{}$, $l_{n}^{(j)}(\om{}) \to \infty$ as $n \to \infty$ for $j=1,2$.
So by Proposition \ref{lemma:Strng law of large numbers}, for almost all $\om{}$, we have $\lim_{n \to \infty} s_{l_n^{(j)}}(\om{})=\kr$ 
 $(j=1,2)$ and hence by Lemma \ref{lemma:l_n^1 l_n^2}
 \begin{align*}
 \underline{t}_r(\om{})=\liminf_{n \to \infty} t_{n,r}(\om{}) \geq \liminf_{n \to \infty}s_{l_n^{(1)}}(\om{}) =\kr,\\ \overline{t}_r(\om{})=\limsup_{n \to \infty} t_{n,r}(\om{}) \leq \limsup_{n \to \infty}s_{l_n^{(2)}}(\om{})=\kr,
 \end{align*}
 which implies $\underline{t}_r(\om{})=\overline{t}_r(\om{})=\kr$. So by Proposition \ref{lemma:Dr=tr}, it follows that for almost all $\om{}$, $D_r(\mu_{\om{}})=\kr.$
\end{proof}

From Propositions \ref{lemma:Strng law of large numbers} and \ref{prop: D_r=k_r}, it follows that the quantization dimension of $\mu_{\om{}}$ (which are not necessarily self-similar) can be approximated almost surely by quantization dimensions of self-similar measures, provided the underlying RIFS satisfies the UESSC.

\section{Conclusion}
In this work, we have rigorously established the quantization dimension for a class of $1$-variable random self-similar measures. By lifting the dynamics to the symbolic space, we resolved the difficulties posed by non-uniform geometric scaling. Our main result (Theorem \ref{thm: main}) confirms that the quantization dimension is determined by the root of the \textit{expected topological pressure}, generalizing the classical deterministic results of Graf and Luschgy. This finding reinforces the link between quantization theory and the thermodynamic formalism of random dynamical systems. Future work may extend this ergodic approach to the more complex settings like general random self-affine or graph-directed systems, where the pressure functions may involve matrix products.

\section*{Declarations}

\begin{itemize}
\item Funding: 
\begin{enumerate}
    \item A. Banerjee acknowledges the Council of Scientific \& Industrial Research (CSIR), India, for the financial support under the scheme “JRF” (File No. 08/0155(12963)/2022-EMR-I).
     \item A. Hossain acknowledges the Council of Scientific \& Industrial Research (CSIR), India, for the financial support under the scheme “JRF” (File No. 08/155(0065)/2019-EMR-I).
    \item Md. N. Akhtar acknowledges the Department of Science and Technology (DST), Govt. of India, for the financial support under the scheme “Fund for Improvement of S\&T Infrastructure (FIST)” (File No. SR/FST/MS-I/2019/41).
    \end{enumerate}

\item Conflict of interest/Competing interests (check journal-specific guidelines for which heading to use):  All authors certify that they have no affiliations with or involvement in any organization or entity with any financial or non-financial interest in the subject matter or materials discussed in this manuscript.

\item Ethics approval and consent to participate: This work did not contain any studies involving animal or human participants, nor did it occur in any private or protected areas. No specific permissions were required for corresponding locations.

\item Consent for publication: All authors have given consent to publish this work.

\item Data availability: No Data associated in the manuscript.

\item Materials availability: Not applicable.

\item Code availability: Not applicable.

\item Author contribution: All the authors contributed equally to this work.
\end{itemize}

\bibliographystyle{abbrv}
\bibliography{sn-bibliography}

\end{document}